\documentclass{amsart}
\usepackage{enumitem}
\usepackage{amsmath}
\usepackage{amssymb}
\usepackage{centernot}
\usepackage{mathtools}
\usepackage{xcolor}

\makeatletter
\@namedef{subjclassname@2020}{\textup{2020} Mathematics Subject Classification}
\makeatother

\newcommand{\pres}[3]{\textnormal{#1} \langle #2 \mid #3 \rangle}

\newcommand{\lra}[1]{\xleftrightarrow{\ast}_{#1}}
\newcommand{\xra}[1]{\xrightarrow{}^\ast_{#1}}
\newcommand{\xr}[1]{\xrightarrow{}_{#1}}

\newcommand{\trev}{\text{rev}}
\newcommand{\CF}{\mathcal{C}_{\operatorname{cf}}}

\newcommand{\DCF}{\mathcal{C}_{\operatorname{dcf}}}
\newcommand{\REG}{\mathcal{C}_{\operatorname{reg}}}
\newcommand{\IND}{\mathcal{C}_{\operatorname{ind}}}

\newcommand{\OC}{\mathcal{C}_{\operatorname{1C}}}
\newcommand{\cc}{\mathcal{C}}
\newcommand{\ct}{\mathcal{T}}
\newcommand{\cR}{\mathcal{R}}

\newcommand{\cW}{\mathcal{W}}
\newcommand{\cI}{\mathcal{I}}

\DeclareMathOperator{\AFL}{AFL}

\DeclareMathOperator{\IP}{IP}
\DeclareMathOperator{\WP}{WP}

\newtheorem{theorem}{Theorem}[section] 
\newtheorem*{theorems}{Theorem}
\newtheorem{lemma}[theorem]{Lemma}     
\newtheorem{corollary}[theorem]{Corollary}
\newtheorem*{corollarys}{Corollary}
\newtheorem{proposition}[theorem]{Proposition}

\theoremstyle{definition}
\newtheorem{definition}{Definition}
\newtheorem{question}{Question}
\newtheorem{example}{Example}
\newtheorem{remark}{Remark}

\numberwithin{equation}{section}

\title[The Word Problem for Free Products]
 {On The Word Problem for Free Products of Semigroups and Monoids} 

\subjclass[2020]{20M05 (primary) 20F10, 20M35, 68Q45 (secondary)}

\keywords{}

\author{Carl-Fredrik Nyberg-Brodda}
\address{Department of Mathematics, Alan Turing Building, University of Manchester, UK.}
\email{carl-fredrik.nybergbrodda@manchester.ac.uk}
\thanks{The author is currently a research associate, funded by the Dame Kathleen Ollerenshaw Trust, at the University of Manchester, UK}

\date{today}

\date{\today}

\begin{document}

\begin{abstract}
We study the language-theoretic aspects of the word problem, in the sense of Duncan \& Gilman, of free products of semigroups and monoids. First, we provide algebraic tools for studying classes of languages known as super-$\AFL$s, which generalise e.g. the context-free or the indexed languages. When $\cc$ is a super-$\AFL$ closed under reversal, we prove that the semigroup (monoid) free product of two semigroups (resp. monoids) with word problem in $\cc$ also has word problem in $\cc$. This recovers and generalises a recent result by Brough, Cain \& Pfeiffer that the class of context-free semigroups (monoids) is closed under taking free products. As a group-theoretic corollary, we deduce that the word problem of the (group) free product of two groups with word problem in $\cc$ is also in $\cc$. As a particular case, we find that the free product of two groups with indexed word problem has indexed word problem. 
\end{abstract}

\maketitle

\noindent An{\=\i}s{\=\i}mov \cite{Anisimov1971} introduced language-theoretic methods to group theory, by studying the language, over some finite generating set, of words equal to the identity element $1$ in a group $G$. The realisation that this language -- called the \textit{word problem} for $G$ -- can encode significant algebraic properties of $G$ led to the introduction of entirely new methods to group theory. Two of the first results proved in this new area, both proved by An{\=\i}s{\=\i}mov, are (1) a group $G$ has regular word problem if and only if  $G$ is finite; and (2) the free product $G = G_1 \ast G_2$ of two groups $G_1, G_2$ with context-free word problem also has context-free word problem. A decade later, Muller \& Schupp \cite{Muller1983, Dunwoody1985} proved a deep and now famous theorem, greatly extending An{\=\i}s{\=\i}mov's initial work on context-free groups. Specifically, they proved that a group has context-free word problem if and only if it is virtually free, i.e. has a finite index free subgroup. This demonstrated in a very clear fashion the deep connections between formal language theory and group theory.

It is only natural to wish to emulate such a success in the theory of monoids and semigroups. One is, however, immediately faced with a difficulty: the language of words representing the identity element in a monoid can often be entirely trivial, even in cases when the monoid is rather complicated.\footnote{For example, if $M = \pres{Mon}{a,b}{abaab =a}$, then clearly no non-empty word represents the identity element in $M$, but the decidability of the word problem for $M$ was for a short time an open problem, posed by Howie \&  Pride (see \cite[Example~7]{Howie1986}, later solved by Jackson \cite{Jackson1986}).} An alternative definition is hence needed. For a group $G$, two words $u, v$, written over the generators of $G$, are equal in $G$ if and only if $uv^{-1}$ is an element of the word problem of $G$. Note how we obtain $v^{-1}$ from the word $v$: we replace each generator $a$ by $a^{-1}$, and then reverse the word. The former operation is language-theoretically trivial; reversal, on the other hand, is not, and hence seems an important component of the word problem. Indeed, recognising this, Duncan \& Gilman \cite[Definition~5.1]{Duncan2004} defined, for a monoid $M$ generated by a finite set $A$, the language
\begin{equation}
\WP_A^M := \{ u \# v^\trev \mid u, v \in A^\ast, u =_M v\}, \tag{\ref{Eq:WP_A^M_def}}
\end{equation}
where $\#$ is some fixed symbol not in $A$, and $w^\trev$ denotes the \textit{reversal} of the word $w$, i.e. the word $w$ read backwards. The language \eqref{Eq:WP_A^M_def} is called the \textit{word problem} for $M$. This has many natural properties; for example, free monoids have context-free word problem. Indeed, the language is sufficiently natural to have been studied independently two decades prior to Duncan \& Gilman's article, see \cite[Corollary~3.8]{Book1982b}. The language \eqref{Eq:WP_A^M_def} has been studied by several authors \cite{Holt2008, Hoffmann2012, Brough2018, Brough2019, NybergBrodda2020c, NybergBrodda2020b, NybergBrodda2021d}. The approach via studying this language has also led to new insights in the group case; for example, a new proof of Herbst's result (see \cite{Herbst1991, Herbst1993}) that a group has one-counter word problem if and only if it is virtually free \cite{Holt2008}, which avoids the deep results involved in proving the Muller-Schupp theorem. 

In this article, we will study some further properties of the word problem for semigroups and monoids. In particular, we will focus on closure properties under taking (semigroup or monoid) free products, and for which classes of languages $\cc$ such closure can be obtained. We will investigate this question when $\cc$ is a \textit{super-$\AFL$}, as introduced by Greibach \cite{Greibach1970}. Such classes generalise the context-free languages and the indexed languages; roughly speaking, an $\AFL$ is a super-$\AFL$ if it is closed under ``recursion'' (see \S\ref{Sec:Toolbox} for details). Specifically, we show (Theorem~\ref{Theo:SGP_closed_under_FP} and Theorem~\ref{Thm:Monoid_FP_WP}) the following:

\begin{theorems}
Let $\cc$ be a super-$\AFL$ closed under reversal. Then the class of semigroups (monoids) with word problem in $\cc$ is closed under taking semigroup (monoid) free products.
\end{theorems}

When $\cc = \CF$, the class of context-free languages, this recovers a recent result by Brough, Cain \& Pfeiffer \cite{Brough2019}, who proved their result directly using pushdown automata. The proof in this article of their result is, on the other hand, independent on any particular specification of the context-free languages, and only depends on certain closure properties satisfied by this class. We find another example of a group-theoretic result which falls out of monoid-theoretic considerations. Indeed, as the monoid free product coincides with the group free product for the class of groups, we find (Corollary~\ref{Cor:WP_groups_closed_FP}) the following:

\begin{corollarys}
Let $\cc$ be a super-$\AFL$ closed under reversal. Then the class of groups with word problem in $\cc$ is closed under free products.
\end{corollarys}

As a particular corollary, we find (Corollary~\ref{Cor:Indexed_groups_WP}) that the class of groups with indexed word problem is closed under free products. There does not appear to be a direct group-theoretic (or otherwise) proof of this result in the literature, unlike Herbst's result mentioned above, which was first proved by group-theoretic techniques and then reproved via monoid theory. 

An overview of the article is as follows. In \S\ref{Sec:Background}, we provide some background and notation used throughout the article. In \S\ref{Sec:Toolbox} we give three tools for studying the language theory associated to the word problem for a free product. Finally, in \S\ref{Sec:FPandWP}, we use these tools to prove the main results of this article. 
 
\clearpage
\section{Background}\label{Sec:Background}

\noindent We assume the reader is familiar with the fundamentals of formal language theory. In particular, an $\AFL$ (\textit{abstract family of languages}) is a class of languages closed under homomorphism, inverse homomorphism, intersection with regular languages, union, concatenation, and the Kleene star. For some background on this, and other topics in formal language theory, we refer the reader to standard books on the subject \cite{Harrison1978, Hopcroft1979, Berstel1985}. The paper also assumes familiarity with the basics of the theory of semigroup, monoid, and group presentations, which will be written as $\pres{Sgp}{A}{\cR}$, $\pres{Mon}{A}{\cR}$, and $\pres{Gp}{A}{\cR}$, respectively. For further background see e.g. \cite{Adian1966,Magnus1966,Neumann1967, Lyndon1977,Campbell1995}.

\subsection{Basic notation}

Let $A$ be a finite alphabet, and let $A^\ast$ denote the free monoid on $A$, with identity element denoted $\varepsilon$ or $1$, depending on the context. Let $A^+$ denote the free semigroup on $A$, i.e. $A^+ = A^\ast - \{ \varepsilon\}$. For $u, v \in A^\ast$, by $u \equiv v$ we mean that $u$ and $v$ are the same word. For $w \in A^\ast$, we let $|w|$ denote the \textit{length} of $w$, i.e. the number of letters in $w$. We have $|\varepsilon| = 0$. If $w \equiv a_1 a_2 \cdots a_n$ for $a_i \in A$, then we let $w^\trev$ denote the \textit{reverse} of $w$, i.e. the word $a_n a_{n-1} \cdots a_1$. Note that $^\trev \colon A^\ast \to A^\ast$ is an anti-homomorphism, i.e. $(uv)^\trev \equiv v^\trev u^\trev$ for all $u, v \in A^\ast$. If $X \subseteq A^\ast$, then we let $X^\trev = \{ x^\trev \mid x \in X \}$. If the words $u, v \in A^\ast$ are equal in the monoid $M = \pres{Mon}{A}{\cR}$, then we denote this $u =_M v$. By $M^\trev$ we mean the \textit{reversed} monoid
\[
M^\trev = \pres{Mon}{A}{ \{ u^\trev = v^\trev \mid (u,v) \in \cR \}}.
\]
Of course, the word problem for $M$ is decidable if and only if it is decidable for $M^\trev$, by virtue of the fact that $u =_M v$ if and only if $u^\trev =_{M^\trev} v^\trev$. Finally, when we say that a monoid $M$ is generated by a finite set $A$, we mean that there exists a surjective homomorphism $\pi \colon A^\ast \to M$.

We give some notation for rewriting systems. For an in-depth treatment and further explanations of the terminology, see e.g. \cite{Book1982, Jantzen1988, Book1993}. A \textit{rewriting system} $\cR$ on $A$ is a subset of $A^\ast \times A^\ast$. An element of $\cR$ is called a \textit{rule}. The system $\cR$ induces several relations on $A^\ast$. We will write $u \xr{\cR} v$ if there exist $x, y \in A^\ast$ and a rule $(\ell, r) \in \cR$ such that $u \equiv x\ell y$ and $v \equiv xry$. We let $\xra{\cR}$ denote the reflexive and transitive closure of $\xr{\cR}$. We denote by $\lra{\cR}$ the symmetric, reflexive, and transitive closure of $\xr{\cR}$. The relation $\lra{\cR}$ defines the least congruence on $A^\ast$ containing $\cR$. For $X \subseteq A^\ast$, we let $\langle X \rangle_\cR$ denote the set of \textit{ancestors} of $X$, i.e. $\langle X \rangle_\cR = \{ w \in A^\ast \mid \exists x \in X \textnormal{ such that } w \xra{\cR} x \}$. The monoid $\pres{Mon}{A}{\cR}$ is identified with the quotient $A^\ast / \lra{\cR}$. For a rewriting system $\ct \subseteq A^\ast \times A^\ast$ and a monoid $M = \pres{Mon}{A}{\cR}$, we say that $\ct$ is $M$\textit{-equivariant} if for every rule $(u, v) \in \ct$, we have $u =_M v$. That is, $\ct$ is $M$-equivariant if and only if $\lra{\ct} \subseteq \lra{\cR}$.

 Let $u, v \in A^\ast$ and let $n \geq 0$. If there exist words $u_0, u_1, \dots, u_n \in A^\ast$ such that 
\[
u \equiv u_0 \xr{\cR} u_1 \xr{\cR} \cdots \xr{\cR} u_{n-1} \xr{\cR} u_n \equiv v,
\]
then we denote this $u \xr{\cR}^n v$, i.e. $u$ rewrites to $v$ in $n$ steps. Thus $\xra{\cR} = \bigcup_{n \geq 0} \xr{\cR}^n$.

A rewriting system $\cR \subseteq A^\ast \times A^\ast$ is said to be \textit{monadic} if $(u, v) \in \cR$ implies $|u| \geq |v|$ and $v \in A \cup \{ \varepsilon \}$. We say that $\cR$ is \textit{special} if $(u, v) \in \cR$ implies $v \equiv \varepsilon$. Every special system is monadic. Let $\cc$ be a class of languages. A monadic rewriting system $\cR$ is said to be $\cc$ if for every $a \in A \cup \{ \varepsilon \}$, the language $\{ u \mid (u, a) \in \cR \}$ is in $\cc$. Thus, we may speak of e.g. monadic $\cc$-rewriting systems or monadic context-free rewriting systems.

\subsection{Word problems}

Let $G$ be a group with finite (group) generating set $A$, with $A^{-1}$ the set of inverses of the generators $A$. The language 
\begin{equation}\label{Eq:IP_group}
\IP_{A \cup A^{-1}}^G := \{ w \mid w  \in (A \cup A^{-1})^\ast, w =_G 1 \}
\end{equation}
is called the (group-theoretic) \textit{word problem} for $G$, see e.g. \cite{Anisimov1971,Muller1983}. Let $M$ be a monoid with a finite generating set $A$. Translating the above definition of the word problem directly to $M$ does not, in general, yield much insight into the structure of $M$. That is, the language 
\begin{equation}\label{Eq:IP_Defn}
\IP_A^M = \{ w \in A^\ast \mid w =_M 1 \},
\end{equation}
does not encode much of the structure of $M$ (with some notable exceptions). Duncan \& Gilman \cite[p.\ 522]{Duncan2004}, realising this, introduced a different language. The (monoid) \textit{word problem of $M$ with respect to $A$} is defined as the language
\begin{equation}\label{Eq:WP_A^M_def}
\WP_A^M := \{ u \# v^\trev \mid u, v \in A^\ast, u =_M v\},
\end{equation}
where $\#$ is some fixed symbol not in $A$. For a class of languages $\cc$, we say that $M$ has $\cc$-word problem if $\WP_A^M$ is in $\cc$. If $\cc$ is closed under inverse homomorphism, then $M$ having $\cc$-word problem does not depend on the finite generating set $A$ chosen for $M$ (see e.g. \cite{Hoffmann2012}). In particular, if $\cc$ is an $\AFL$, we will speak of monoids or semigroups having word problem in $\cc$ without any reference to any particular finite generating set. Furthermore, if $M$ is a group, then $M$ has group-theoretic word problem in $\cc$ if and only if $\WP_A^M$ is in $\cc$ \cite[Theorem~3]{Duncan2004}. In particular, if $\cc$ is a class of languages closed under inverse homomorphism, and $M$ is a group generated by a finite set $A$ (as a monoid), then one of \eqref{Eq:IP_group}, \eqref{Eq:IP_Defn}, and \eqref{Eq:WP_A^M_def} lies in $\cc$ if and only if they all do.

Of course, for a semigroup $S$ generated by a finite set $A$, by which we mean there exists a surjective homomorphism $\pi \colon A^+ \to S$, one may also define the semigroup-theoretic word problem 
\begin{equation}\label{Eq:SGPWP}
\operatorname{SWP}_A^S := \{ u \# v^\trev \mid u, v \in A^+ \textnormal{ and } u =_S v\}.
\end{equation}
We will, however, for ease of notation and as context will always make it clear, be somewhat abusive, and simply use the notation $\WP_A^S$ for the language \eqref{Eq:SGPWP}.

\subsection{Free products}\label{Subsec:Free_products}

Let $X \colon \mathbb{N} \to \{ 1, 2 \}$ be a parametrisation such that either $X(2j) = 1$ and $X(2j+1) = 2$, or else $X(2j) = 2$ and $X(2j+1) = 1$, for all $ j \in \mathbb{N}$. A parametrisation $X$ of this form will be called \textit{standard}, and will be used throughout this article.

Let $S_1 = \pres{Sgp}{A_1}{R_1}$ and $S_2 = \pres{Sgp}{A_2}{R_2}$ be two semigroups, where $A_1 \cap A_2 = \varnothing$. Then the \textit{semigroup free product} $S_1 \ast S_2$ of $S_1$ and $S_2$ is defined as the semigroup with presentation $\pres{Sgp}{A_1 \cup A_2}{R_1 \cup R_2}$. Semigroup free products may appear unnatural to a group theorist; for example, the free product of two trivial semigroups (where the trivial semigroup is the semigroup $\pres{Sgp}{e}{e^2=e}$ with one element) is infinite, and turns out to be a semigroup with rather interesting properties, appearing in various places, see e.g. \cite{Benzaken1975, Campbell1996b, Cain2009c, Zelenyuk2009, Shneerson2017}. Furthermore, even if $S_1, S_2$ are both monoids, i.e. have identity elements, then their respective identity elements need not be an identity element of their free product. In particular, the semigroup free product of groups need not be a group. On the other hand, a normal form lemma for semigroup free products is very simple to state. 

\begin{lemma}[{See e.g. Howie \cite[\S8.2]{Howie1995}}]\label{Lem:Semigroup_free_product_normalforms}
Let $S_1, S_2$ be two semigroups, generated by disjoint sets $A_1, A_2$, respectively. Let $S = S_1 \ast S_2$ denote their semigroup free product. Then for $u, v \in (A_1 \cup A_2)^+$ we have $u =_S v$ if and only if there exist unique $u_1, v_1, u_2, v_2, \dots, u_n, v_n$ such that
\begin{enumerate}[label=(\arabic*)]
\item $u \equiv u_1 u_2 \cdots u_n$ and $v \equiv v_1 v_2 \cdots v_n$;
\item $u_i, v_i \in A_{X(i)}^+$ and $u_i =_{S_{X(i)}} v_i$ for all $1 \leq i \leq n$,
\end{enumerate}
where $X$ is a standard parametrisation. 
\end{lemma}

Let $M_1 = \pres{Mon}{A_1}{R_1}$ and $M_2 = \pres{Mon}{A_2}{R_2}$ be two monoids, with $A_1 \cap A_2 = \varnothing$. The \textit{monoid free product} $M_1 \ast M_2$ of $M_1$ and $M_2$ is defined as the monoid with presentation $\pres{Mon}{A_1 \cup A_2}{R_1 \cup R_2}$. Thus the monoid free product identifies the identity elements of the free factors, the free product of two trivial monoids is trivial, and the monoid free product of groups coincides with the (group) free product of the groups (see e.g. \cite[p. 266]{Howie1995}). We note in passing the obvious fact that $(M_1 \ast M_2)^\trev = M_1^\trev \ast M_2^\trev$.

Let $u \equiv u_0 u_1 \cdots u_n$ be a word in $(A_1 \cup A_2)^\ast$ such that $u_i \in A_{X(i)}^\ast$ for all $0 \leq i \leq n$, where $X$ is a standard parametrisation. Obviously, any word in $(A_1 \cup A_2)^\ast$ can be written uniquely as such a product. We will therefore make little distinction between the word $u$ and its factorisation. We say that (the factorisation of) $u$ is \textit{reduced} if it is empty, or else $u_i \neq_{M_{X(i)}} 1$ for all $0 \leq i \leq n$. Obviously, by removing subwords equal to $1$, for every $u$ as above there exists $u' \in (A_1 \cup A_2)^\ast$ such that $u =_{M_1 \ast M_2} u'$ and such that $u' \equiv u_0' u_1' \cdots u_m'$ is reduced. We call any such $u'$ a \textit{reduced form} of  $u$. Reduced forms need not be unique. However, any reduced form $u'$ of $u$ clearly only has one factorisation $u_0' u_1' \cdots u_m'$ such that this product is reduced. We call this factorisation $u_0' u_1' \cdots u_m'$ the \textit{normal form} of the reduced form $u'$. We then have the following:

\begin{lemma}[{See e.g. Howie \cite[\S8.2]{Howie1995}}]\label{Lem:Monoid_free_product_normal_form}
Let $M_1, M_2$ be two monoids, generated by disjoint sets $A_1, A_2$, respectively. Let $M = M_1 \ast M_2$ denote their monoid free product. Let $u, v \in (A_1 \cup A_2)^\ast$, and let $u'$ resp. $v'$ be any reduced forms for $u$ resp. $v$, with normal forms
\[
u' \equiv u_0 u_1 \cdots u_m \quad \textnormal{resp.} \quad v' \equiv v_0 v_1 \cdots v_n.
\]
Then $u =_M v$ if and only if 
\begin{enumerate}
\item $n=m$, and
\item $u_i, v_i \in A_{X(i)}^\ast$ and $u_i =_{M_{X(i)}} v_i$ for all $0 \leq i \leq n$,
\end{enumerate}
where $X$ is a standard parametrisation.
\end{lemma}

We shall presently introduce language-theoretic tools to interpret Lemma~\ref{Lem:Monoid_free_product_normal_form}.

\clearpage
\section{A linguistic toolbox}\label{Sec:Toolbox}

\noindent This present section is divided into three parts. In the first, we provide a new, more combinatorial, definition of \textit{super-}$\AFL$s, first introduced by Greibach \cite{Greibach1970}. In the second, we study \textit{alternating products} of languages of a certain form, and prove that these products preserve language-theoretic properties of their factors. Informally, alternating products are the language-theoretic equivalent of semigroup free products. Finally, in the third part, we study \textit{bipartisan ancestors}, and prove further preservation properties. Informally, bipartisan ancestors are a language-theoretic method for dealing with the fact that the identity elements of the factors are identified in a monoid free product. 

\subsection{Super-$\AFL$s}

The theorems in this paper involve a special type of classes of languages, called \textit{super-$\AFL$s}. These were introduced by Greibach \cite{Greibach1970}. The defining property of these classes of languages is somewhat cumbersome to state directly. We follow Book, Jantzen \& Wrathall \cite{Book1982b} in this.

Let $A$ be an alphabet. For each $a \in A$, let $\sigma(a)$ be a language (over any finite alphabet); let $\sigma(\varepsilon) = \{ \varepsilon\}$; for every $x, y \in A^\ast$ let $\sigma(xy) = \sigma(x)\sigma(y)$; and for every $L \subseteq A^\ast$, let $\sigma(L) = \bigcup_{w\in L} \sigma(w)$. We then say that $\sigma$ is a \textit{substitution}. For a class $\cc$ of languages, if for every $a \in A$ we have $\sigma(a) \in \cc$, then we say that $\sigma$ is a $\cc$-\textit{substitution}. Let $A$ be an alphabet, and $\sigma$ a substitution on $A$. For every $a \in A$, let $A_a$ denote the smallest finite alphabet such that $\sigma(a) \subseteq A_a^\ast$. Extend $\sigma$ to $A \cup (\bigcup_{a \in A}A_a)$ by defining $\sigma(b) = \{ b \}$ whenever $b \in (\bigcup_{a \in A} A_a) \setminus A$. For $L \subseteq A^\ast$, let $\sigma^1(L) = \sigma(L)$, and let $\sigma^{n+1}(L) = \sigma(\sigma^n(L))$ for $n \geq 1$. Let $\sigma^\infty(L) = \bigcup_{n> 0} \sigma^n(L)$. Then we say that $\sigma^\infty$ is an \textit{iterated substitution}. If for every $b \in A \cup (\bigcup_a A_a)$ we have $b \in \sigma(b)$, then we say that $\sigma^\infty$ is a \textit{nested} iterated substitution. If $\sigma^\infty$ is nested, then it is convenient for inductive purposes to set $\sigma^0(L) := L$. Note that the nested property ensures $L \subseteq \sigma(L)$, so $\bigcup_{n \geq 0} \sigma^n(L) = \bigcup_{n>0} \sigma^n(L)$. We say that $\cc$ is closed under nested iterated substitution if for every $\cc$-substitution $\sigma$ and every $L \in \cc$, we have: if $\sigma^\infty$ is a nested iterated substitution, then $\sigma^\infty(L) \in \cc$. 

We will now give a closely related definition, phrased in terms of rewriting systems, which is more combinatorial in nature.

\begin{definition}
Let $\cc$ be a class of languages. Let $\cR \subseteq A^\ast \times A^\ast$ be a rewriting system. Then we say that $\cR$ is $\cc$\textit{-ancestry preserving} if for every $L \subseteq A^\ast$ with $L \in \cc$, we have $\langle L \rangle_\cR \in \cc$. If every monadic $\cc$-rewriting system is $\cc$-ancestry preserving, then we say that $\cc$ has the \textit{monadic ancestor property}.
\end{definition}

\begin{example}
If $\cR \subseteq A^\ast \times A^\ast$ is a monadic context-free rewriting system, and $L \subseteq A^\ast$ is a context-free language, then $\langle L \rangle_\cR$ is a context-free language \cite[Theorem~2.2]{Book1982b}. That is, every monadic $\CF$-rewriting system is $\CF$-ancestry preserving; in other words, the class of context-free languages has the monadic ancestor property. 
\end{example}

We will presently show a straightforward proposition (namely Proposition~\ref{Prop:MAP<=>NIS}), which connects the monadic ancestor property with the somewhat complicated notion of nested iterated substitutions. Before this, we need a slightly technical, but straightforward to prove, lemma, concerning monadic rewriting systems.

\begin{lemma}\label{Lem:Do_erasings_first}
Let $\cR^{(m)} \subseteq A^\ast \times A^\ast$ be a monadic rewriting system. Let $A_1 \subseteq A$, and let $\cR^{(1)} \subseteq A^\ast \times A^\ast$ be the rewriting system with rules $\{ (\varepsilon, a) \mid a \in A_1\}$. Let $\cR = \cR^{(1)} \cup \cR^{(m)}$. Then for any $L \subseteq A^\ast$, we have $\langle L \rangle_{\cR} = \langle \langle L \rangle_{\cR^{(m)}} \rangle_{\cR^{(1)}}$.
\end{lemma}
\begin{proof}
For ease of notation, we denote $\xr{\cR^{(1)}}$ by $\xr{(1)}$; we denote $\xr{\cR^{(m)}}$ by $\xr{(m)}$; and $\xr{\cR}$ by $\xr{(\cup)}$. The notation is extended to $\xra{(1)}, \xra{(m)}, \xra{(\cup)}$, etc.

Let $u, w \in A^\ast$ be such that $w \in \langle u \rangle_{\cR}$. Then $w \xra{(\cup)} u$, so $w \xr{(\cup)}^k u$ for some $k \geq 0$. It suffices to show that there is some $v \in A^\ast$ such that $w \xra{(1)} v \xra{(m)} u$. We prove this claim by induction on $k$. The case $k=0$ is trivial, for then we can take $v \equiv w (\equiv u)$. Suppose $k>0$ and let $u_0, u_1, \dots, u_k \in A^\ast$ be such that 
\begin{equation}\label{Eq:Rewriting_in_technical_lemma}
w \equiv u_0 \xr{(\cup)} u_1 \xr{(\cup)} \cdots \xr{(\cup)} u_{k-1} \xr{(\cup)} u_k \equiv u.
\end{equation}
By the inductive hypothesis, there exists $v_1 \in A^\ast$ such that $u_1 \xra{(1)} v_1 \xra{(m)} u$. If the rewriting $u_0 \xr{(\cup)} u_1$ in \eqref{Eq:Rewriting_in_technical_lemma} is a rewriting $u_0 \xr{(1)} u_1$, then we may take $v \equiv v_1$. Suppose, then, that $u_0 \xr{(m)} u_1$. Let $(s, a_j) \in \cR^{(m)}$ be the rule applied to rewrite $u_0 \xr{(m)} u_1$, where $a_j \in A \cup \{ \varepsilon \}$ and $s \in A^\ast$ is non-empty. Write $u_1 \equiv a_1 a_2 \cdots a_{j-1} a_j a_{j+1} \cdots a_n$ for some $n \geq 0$ (where $n=0$ means $u_1 \equiv \varepsilon$) and $a_i \in A$ for $1 \leq i \leq n, i \neq j$. Then $u_0 \equiv a_1 a_2 \cdots a_{j-1} s a_{j+1} \cdots a_n$. Now, as $u_1 \xra{(1)} v_1$, we have that 
\[
v_1 \equiv s_1 a_1 s_2 a_2 \cdots s_j a_j s_{j+1} \cdots s_n a_n s_{n+1}
\]
where the $s_i$ are such that $\varepsilon \xra{(1)} s_i$ for $1 \leq i \leq n+1$. Let 
\[
v_1' \equiv s_1 a_1 s_2 a_2 \cdots s_j s s_{j+1} \cdots s_n a_n s_{n+1}.
\]
Then $v_1' \xr{(m)} v_1$ by using the rule $(s, a_j)$, so $v_1' \xra{(m)} u$. Furthermore, as $\varepsilon \xra{(1)} s_i$ for every $1 \leq i \leq n+1$, we have $u_0 \xra{(1)} v_1'$. That is, $w \xra{(1)} v_1' \xra{(m)} u$, so we can take $v \equiv v_1'$. 
\end{proof}

We can now proceed with the main proposition of this section. 

\begin{proposition}\label{Prop:MAP<=>NIS}
Let $\cc$ be an $\AFL$. Then $\cc$ is closed under nested iterated substitution if and only if it has the monadic ancestor property.
\end{proposition}
\begin{proof}
The proof of the forward implication is the same, \textit{mutatis mutandis}, as the proof of \cite[Theorem~2.2]{Book1982b}. For the reverse, suppose $\cc$ has the monadic ancestor property. Let $\sigma$ be a $\cc$-substitution such that $\sigma^\infty$ is a nested iterated substitution on the alphabet $A$, and let for every $a \in A$ the language $A_a$ be as earlier. Let $L \subseteq A^\ast$ be such that $L \in \cc$. We must show that $\sigma^\infty(L) \in \cc$. Let $\Sigma = A \cup (\bigcup_{a \in A} A_a)$, and define a rewriting system $\cR_\sigma \subseteq \Sigma^\ast \times \Sigma^\ast$ as the system
\begin{equation}\label{Eq:cr_sigma_rules}
\cR_\sigma = \bigcup_{a \in \Sigma} \bigcup_{w \in \sigma(a)} \{ (w, a) \}.
\end{equation}
This is not, in general, a monadic system, as there can be rules of the form $(\varepsilon, a)$ in $\cR_\sigma$. Partition $\cR_\sigma$ as $\cR_\sigma^{(1)} \cup \cR_\sigma^{(m)}$, where $\cR_\sigma^{(1)}$ is the set of all rules of the form $(\varepsilon, a)$ for some $a \in \Sigma$, and $\cR_\sigma^{(m)} = \cR_\sigma - \cR_\sigma^{(1)}$. Then $\cR_\sigma^{(m)}$ is a monadic rewriting system. Furthermore, $\cR_\sigma^{(m)}$ is a $\cc$-rewriting system, as for any $a \in \Sigma$ the language of left-hand sides of rules with right-hand side $a$ is $\sigma(a)$ for every $a \in \Sigma$, and $\sigma(a) \in \cc$ as $\sigma$ is a $\cc$-substitution. Let $\Sigma_1 \subseteq \Sigma$ be the set of letters $\{ a \mid (\varepsilon, a) \in \cR_\sigma^{(1)}\}$.

Let $\tau \subseteq \Sigma^\ast \times \Sigma^\ast$ be defined by $\tau = \bigcup_{a \in \Sigma_1}((a,\varepsilon)^\ast \cup (a,a)^\ast)$. Then $\tau$ is a rational transduction. Furthermore, for any language $K \subseteq \Sigma^\ast$, it is easy to see that $\langle K \rangle_{\cR_\sigma^{(1)}} = \tau(K)$. That is, informally speaking, any ancestor of a word $w$ under $\cR_\sigma^{(1)}$ can be obtained from $w$ by deleting some number of letters from $\Sigma_1$. 

We claim $\langle L \rangle_{\cR_\sigma}  = \sigma^\infty(L)$. This would complete the proof; indeed, by Lemma~\ref{Lem:Do_erasings_first}, we have $\langle L \rangle_{\cR_\sigma} = \langle \langle L \rangle_{\cR_\sigma^{(m)}} \rangle_{\cR_\sigma^{(1)}}$. Hence
\begin{equation}\label{Eq:L_is_ancestor_superAFL}
\langle L \rangle_{\cR_\sigma} = \langle \langle L \rangle_{\cR_\sigma^{(m)}} \rangle_{\cR_\sigma^{(1)}} = \tau \langle L \rangle_{\cR_\sigma^{(m)}}.
\end{equation}
As $\cR_\sigma^{(m)}$ is a monadic $\cc$-rewriting system, and $\cc$ has the monadic ancestor property, it follows that $\langle L \rangle_{\cR_\sigma^{(m)}} \in \cc$. As $\cc$ is an $\AFL$, it is closed under rational transduction, so the right-hand side, and thereby also the left-hand side $\langle L \rangle_{\cR_\sigma}$, of \eqref{Eq:L_is_ancestor_superAFL} is in $\cc$. We prove the claimed equality. 

$(\subseteq)$ If $w \in \langle L \rangle_{\cR_\sigma}$, there exists $u \in L$ and $n \geq 0$ such that $w \xr{\cR_\sigma}^n u$. We prove that there exists $k \geq 1$ such that $w \in \sigma^k(L)$ by induction on this $n$. This would prove $w \in \sigma^\infty(L)$. If $n=0$, then $w \equiv u$, and as $\sigma^\infty$ is nested we have $L \subseteq \sigma(L) \subseteq \sigma^\infty(L)$, and we are done. Assume $n>0$. Then there exists $w' \in \langle u \rangle_{\cR_\sigma}$ such that $w \xr{\cR_\sigma} w' \xr{\cR_\sigma}^{n-1} u$. By the inductive hypothesis there exist $k'\geq 1$ such that $w' \in \sigma^{k'}(L)$. As $w \xr{\cR_\sigma} w'$, there is a rule $(r, s) \in \cR_\sigma$ such that $w \equiv w_0 r w_1$ and $w' \equiv w_0 s w_1$. But from \eqref{Eq:cr_sigma_rules}, we have $s \in \Sigma$ and $r \in \sigma(s)$. Hence 
\begin{align*}
w \equiv w_0 r w_1 \in w_0 \sigma(s) w_1 \subseteq \sigma(w_0) \sigma(s) \sigma(w_1) = \sigma(w_0 s w_1) = \sigma(w') \subseteq \sigma(\sigma^{k'}(L)).
\end{align*}
Note that the inclusion $\{ x \} \subseteq \sigma(x)$ for $x \in A^\ast$ follows from the fact that $\sigma^\infty$ is nested. As $\sigma(\sigma^{k'}(L)) = \sigma^{k'+1}(L)$, we can take $k = k'+1$. 

$(\supseteq)$ Suppose $w \in \sigma^\infty(L)$. Then there exists $n \geq 0$ such that $w \in \sigma^n(L)$. We prove $w \in \langle L \rangle_{\cR_\sigma}$ by induction on this $n$. If $n=0$, then (by our convention), $w \in L$, so there is nothing to show. Assume $n>0$. Then there exists some $u \in \sigma^{n-1}(L)$ such that $w \in \sigma(u)$. 

By the inductive hypothesis $u \in \langle L \rangle_{\cR_\sigma}$. Write $u \equiv a_1 a_2 \cdots a_k$, where $k \geq 1$ and $a_i \in A$ for $1 \leq i \leq k$. Then $\sigma(u) = \sigma(a_1) \sigma(a_2) \cdots \sigma(a_k)$. Hence $w \equiv w_1 w_2 \cdots w_k$ for some $w_i \in \sigma(a_i)$ for $1 \leq i \leq k$. In particular, $(w_i, a_i) \in \cR_\sigma$ for every $1 \leq i \leq k$. Hence we find 
\[
w \equiv w_1 w_2 \cdots w_k \xra{\cR_\sigma} a_1 a_2 \cdots a_k \equiv u,
\]
so $w \in \langle u \rangle_{\cR_\sigma} \subseteq \langle \langle L \rangle_{\cR_\sigma} \rangle_{\cR_\sigma} = \langle L \rangle_{\cR_\sigma}$, which is what was to be shown. 
\end{proof}

In view of Proposition~\ref{Prop:MAP<=>NIS}, we provide the following definition of a super-$\AFL$.

\begin{definition}\label{Def:super-AFL}
Let $\cc$ be an $\AFL$. Then $\cc$ is said to be a \textit{super-$\AFL$} if it satisfies one of the following two equivalent conditions:
\begin{enumerate}
\item $\cc$ has the monadic ancestor property;
\item $\cc$ is closed under nested iterated substitution.
\end{enumerate}
\end{definition}

Definition~\ref{Def:super-AFL}(2) is the original definition as it was given by Greibach \cite{Greibach1970}. The monadic ancestor property, being directly phrased in terms of rewriting systems, is more directly applicable to our algebraic setting. Throughout the remainder of this article, a super-$\AFL$ will thus be defined by Definition~\ref{Def:super-AFL}(1).

\begin{example} 
The following are examples of super-$\AFL$s:
\begin{enumerate}
\item The class $\CF$ of context-free languages; this is proved via Definition~\ref{Def:super-AFL}(1) in \cite[Theorem~1.2]{Kral1970}, and via Definition~\ref{Def:super-AFL}(2) in \cite[Theorem~2.2]{Book1982b}. 
\item The class $\IND$ of indexed languages \cite{Aho1968, Engelfriet1985}
\end{enumerate}
The following are examples of classes of languages which are \textbf{not} super-$\AFL$s:
\begin{enumerate}\setcounter{enumi}{2}
\item The class $\REG$, e.g. by considering $\cR = \{ (bc, \varepsilon)\}$ and $\langle \varepsilon \rangle_\cR \not\in \REG$.
\item The class $\DCF$ of deterministic context-free languages (it is not an $\AFL$).
\item If $\cc$ is a super-$\AFL$, then $\CF \subseteq \cc$, by \cite[Theorem~2.2]{Greibach1970}. Hence, the class $\OC$ of one-counter languages is not a super-$\AFL$, as $\OC \subset \CF$. More generally, this shows that $\CF$ is the smallest super-$\AFL$.
\end{enumerate}
\end{example}

For more examples and generalisations, we refer the reader to the so-called \textit{hyper-$\AFL$s} defined by Engelfriet \cite{Engelfriet1985}, all of which are super-$\AFL$s.

\subsection{Alternating products}

We will introduce an operation on certain languages, which mimics the operation of the free product of semigroups. Throughout this section, we fix an alphabet $A$ and let $\#$ be a symbol not in $A$. Let $L \subseteq A^\ast \# A^\ast$ be any language. We say that $L$ is \textit{concatenation-closed} (with respect to $\#$) if 
\begin{equation}\label{Eq:Concatenation_closed_def}
u_1 \# v_1 \in L \textnormal{ and } u_2 \# v_2 \in L \implies u_1 u_2 \# v_2 v_1 \in L,
\end{equation}
where $u_1, v_1, u_2, v_2 \in A^\ast$.

\begin{example}
Let $M$ be any finitely generated monoid (or semigroup) with finite generating set $A$. If $u_1 \# v_1, u_2 \# v_2 \in \WP_A^M$, with $u_1, v_1, u_2, v_2 \in A^\ast$ (or $A^+$), then by definition of $\WP_A^M$ we have $u_1 =_M v_1^\trev$ and $u_2 =_M v_2^\trev$, so 
\[
u_1 u_2 =_M v_1^\trev v_2^\trev \equiv (v_2 v_1)^\trev,
\]
from which it follows that
\[
u_1 u_2 \# ((v_2v_1)^\trev)^\trev \equiv u_1 u_2 \# v_2 v_1 \in \WP_A^M.
\]
Hence, $\WP_A^M$ satisfies \eqref{Eq:Concatenation_closed_def}, and it is a concatenation-closed language. 
\end{example}

Given two concatenation-closed languages in $A^\ast \# A^\ast$, we now give an operation for combining them. Let $L_1, L_2 \subseteq A^\ast \# A^\ast$ be concatenation-closed. Then the \textit{alternating product} $L_1 \star L_2$ of $L_1$ and $L_2$ is defined as the language consisting of all words of the form: 
\begin{equation}\label{Eq:Def_alternating}
u_1 u_2 \cdots u_k \# v_k \cdots v_2 v_1,
\end{equation}
where for all $1 \leq i \leq k$ we have $u_i \# v_i \in L_{X(i)}$, with a standard parametrisation $X $ (see \S\ref{Subsec:Free_products}). In particular, $\# \in L_1 \star L_2$ if and only if $\# \in L_1$ or $\# \in L_2$. Alternating products are modelled on the (semigroup) free product, as the following example shows. 

\begin{example}
Let $L_1 = \{ a^n \# a^n \mid n \geq 0 \}$ and $L_2 = \{ b^n \# b^n \mid n \geq 0 \}$. Then, of course, $L_1 = \WP_{\{ a \}}^{\{ a \}^\ast}$ and $L_2 = \WP_{\{ b \}}^{\{ b \}^\ast}$. Now both languages $L_1$ and $L_2$ are concatenation-closed, and it is easy to see that we have
\[
L_1 \star L_2 = \{ a^{n_1} b^{n_2} \cdots  a^{n_k} \# a^{n_k} \cdots b^{n_2} a^{n_1} \mid k \geq 0; n_1, n_2, \dots, n_{k-1} \geq 1; n_k \geq 0 \}.
\]
Thus $L_1 \star L_2 = \{ w \# w^\trev \mid w \in \{ a, b \}^\ast \} = \WP_{\{a,b\}}^{\{a,b\}^\ast}$.
\end{example}

Using the monadic ancestor property, the following gives the main reason for why alternating products are useful. 

\begin{proposition}\label{Prop:L1,L2=>L1starL2_in_C}
Let $\cc$ be a super-$\AFL$. Let $L_1, L_2 \subseteq A^\ast \# A^\ast$ be concatenation-closed languages. Then $L_1, L_2 \in \cc \implies L_1 \star L_2 \in \cc$. 
\end{proposition}
\begin{proof}
Let $\cR_i = \{ (w, \#) \mid w \in L_i\}$ for $i = 1, 2$. By assumption $\cR_1, \cR_2$ are monadic $\cc$-rewriting systems, so $\cR := \cR_1 \cup \cR_2$ is also a monadic $\cc$-rewriting system. As $\cc$ has the monadic ancestor property and $\cc$ contains all singleton languages, we have $L := \langle \# \rangle_\cR \in \cc$. It suffices for the proposition to show that $L = (L_1 \star L_2) \cup \{ \# \}$, as if $\# \in L_1 \star L_2$ we conclude $L_1 \star L_2 \in \cc$; and if $\# \not\in L_1 \star L_2$, then 
\[
L_1 \star L_2 = L \cap (A^\ast \# A^\ast - \{ \# \}),
\]
and as $\cc$ is closed under intersection with regular languages we find $L_1 \star L_2 \in \cc$.

$(\subseteq)$ Suppose $w \in (L_1 \star L_2) \cup \{ \# \}$. As $\# \in L$, assume $w \in L_1 \star L_2$ and write $w \equiv u_1 u_2 \cdots u_k \# v_k \cdots v_2 v_1$ with $u_i \# v_i \in L_{X(i)}$, with $X$ a standard parametrisation. We prove $w \in L$ by induction on this $k$. If $k=0$, then $w \equiv \#$, and there is nothing to show. Suppose $k>0$. As $u_k \# v_k \in L_{X(k)}$, thus $(u_k \# v_k, \#) \in \cR_{X(k)}$, so 
\begin{equation}\label{Eq:w->u_k-1}
w \xr{\cR} u_1 u_2 \cdots u_{k-1} \# v_{k-1} \cdots v_1 v_1.
\end{equation}
By the inductive hypothesis, the right-hand side of \eqref{Eq:w->u_k-1} lies in $\langle \# \rangle_{\cR}$. Thus also $w \in \langle \# \rangle_{\cR} = L$.

$(\supseteq)$ Suppose $w \in L$. Then $w \xr{\cR}^k \#$ for some $k \geq 0$. We prove the claim by induction on $k$. If $k=0$, then $w \equiv \#$, and there is nothing to show. Suppose $k>0$ and that the claim is true for any word rewriting to $\#$ in fewer than $k$ steps. As $w \xr{\cR}^k \#$ and $\cR$ rewrites words in $A^\ast \# A^\ast$ to words in $A^\ast \# A^\ast$, there is some $w' \in A^\ast \# A^\ast$ such that $w \xr{\cR} w' \xr{\cR}^{k-1} \#$. By the inductive hypothesis, $w' \in (L_1 \star L_2) \cup \{ \# \}$. Thus we can write $w \equiv u_1 u_2 \cdots u_n \# v_n \cdots v_2 v_1$, where either all $u_i, v_i$ are empty and $n=0$, or else $u_i \# v_i \in L_{X(i)}$ for all $1 \leq i \leq n$, where $X$ is a standard parametrisation. Let $(u' \# v', \#)$ be the rule of $\cR$ by which $w$ is rewritten to $w'$. As $w'$ contains exactly one occurrence of $\#$, we thus have 
\begin{equation}\label{Eq:w_alternating_proof}
w \equiv u_1 u_2 \cdots u_n u' \# v' v_n \cdots v_2 v_1
\end{equation}
As $(u' \#v', \#) \in \cR$, either $u' \# v' \in L_{X(n)}$ or $u' \# v' \in L_{X(n+1)}$. In the former case, the right-hand side in \eqref{Eq:w_alternating_proof} is an alternating product of elements from $L_1$ and $L_2$, i.e. $w \in L_1 \star L_2$. In the latter case, as $L_{X(n)}$ is concatenation-closed we have $u_n u' \# v' v_n \in L_{X(n)}$, so the right-hand side in \eqref{Eq:w_alternating_proof} is again an alternating product of elements from $L_1$ and $L_2$, i.e. $w \in L_1 \star L_2 \subseteq (L_1 \star L_2) \cup \{ \# \}$. 
\end{proof}

As mentioned, we will use alternating product to model semigroup free products. To deal with monoid free products (and, by extension, also group free products), we need to introduce another language-theoretic tool. 

\subsection{Bipartisan ancestors}

Let $\cR_1, \cR_2 \subseteq A^\ast \times A^\ast$ be two rewriting systems. Let $L \subseteq A^\ast \# A^\ast$. Then we define the \textit{bipartisan $(\cR_1, \cR_2)$-ancestor} of $L$ as the language 
\begin{equation}\label{Eq:Bipartisan_ancestor_def}
L^{\cR_1, \cR_2} = \{ w_1 \# w_2 \mid \exists u_1 \# u_2 \in L \textnormal{ such that } w_i \in \langle u_i \rangle_{\cR_i} \textnormal{ for $i=1, 2$}\}.
\end{equation}
Intuitively, this definition tells us that we can manipulate both the left and the right (whence the name \textit{bipartisan}) sides of $\#$ in words from $L$ using $\cR_1$ resp. $\cR_2$, but that these manipulations are independent of one another. We give an example of this independence. 

\begin{example}
Let $A = \{ a, b\}$, and let $L = \{ a \# a \}$. Let $\cR_1$ be the rewriting system with the rules $(b^n, a)$ for all $n \geq 1$. Then 
\[
L^{\cR_1, \cR_1} = \{ b^{n_1} \# b^{n_2} \mid n_1, n_2 \geq 1\} \cup \{ a\#a \}.
\]
In particular, we have $b^{n_1} \# b^{n_2} \in L^{\cR_1, \cR_1}$ even if $n_1 \neq n_2$.
\end{example}

Our main goal is to prove a general preservation property of taking bipartisan ancestors (Proposition~\ref{Prop:ancestry_main_prop}). Before we can do this, we will first note a weak form of this preservation, which will be used in the subsequent proof of the general result. 

\begin{lemma}\label{Lem:Anc_partial_pres}
Let $\cc$ be a class of languages. Let $\cR_1 \subseteq A_1^\ast \times A_1^\ast$ and $\cR_2 \subseteq A_2^\ast \times A_2^\ast$ be rewriting systems over alphabets $A_1, A_2$ with $A_1 \cap A_2 = \varnothing$. If $\cR_1$ and $\cR_2$ are $\cc$-ancestry preserving, then for every $L \subseteq A_1^\ast \# A_2^\ast$, we have $L \in \cc \implies L^{\cR_1, \cR_2} \in \cc$.
\end{lemma}
\begin{proof}
We claim $L^{\cR_1, \cR_2} = \langle \langle L \rangle_{\cR_1} \rangle_{\cR_2}$. As $\cR_1, \cR_2$ are both $\cc$-ancestry preserving, this is sufficient to prove the lemma. The inclusion $\subseteq$ is immediate by \eqref{Eq:Bipartisan_ancestor_def}. For $\supseteq$, suppose $w \in \langle \langle L \rangle_{\cR_1} \rangle_{\cR_2}$. Then $w \equiv w_1 \# w_2$, and there is some $v_1 \# v_2 \in \langle L \rangle_{\cR_1}$ with $w_1 \# w_2 \xra{\cR_2} v_1 \# v_2$. Furthermore, there is some $u_1 \# u_2 \in  L$ with $v_1 \# v_2 \xra{\cR_1} u_1 \# u_2$. As $u_1 \# u_2 \in L$, we have $u_i \in A_i^\ast$. Thus, as $\cR_1 \subseteq A_1^\ast \times A_1^\ast$ and $A_1 \cap A_2 = \varnothing$, it follows that the rewriting $v_1 \# v_2 \xra{\cR_1} u_1 \# u_2$ effects a rewriting $v_1 \xra{\cR_1} u_1$, and that $v_2 \equiv u_2$. By an analogous argument, the rewriting $w_1 \# w_2 \xra{\cR_2} v_1 \# v_2 \equiv v_1 \# u_2$ effects a rewriting $w_2 \xra{\cR_2} u_2$. Hence $w \equiv w_1 \# w_2$ is such that there is a $u_1 \# u_2 \in L$ with $w_i \xra{\cR_i} u_i$ for $i=1,2$. We conclude from \eqref{Eq:Bipartisan_ancestor_def} that $w \in L^{\cR_1, \cR_2}$. Of course, by symmetry, also $L^{\cR_1, \cR_2} = \langle \langle L \rangle_{\cR_2} \rangle_{\cR_1}$.
\end{proof}
\begin{remark}
In the setting of Lemma~\ref{Lem:Anc_partial_pres}, we have $L^{\cR_1, \cR_2} = \langle \langle L \rangle_{\cR_1} \rangle_{\cR_2}$. However, $L^{\cR_2, \cR_1} = L$, as no rewritings can now be performed in \eqref{Eq:Bipartisan_ancestor_def}. In particular, we can certainly have $L^{\cR_1, \cR_2} \neq L^{\cR_2, \cR_1}$.
\end{remark}

To generalise the above situation to not necessarily disjoint alphabets, let $A$ be an alphabet, and $A_\ell, A_r$ be two alphabets in bijective correspondence with $A$, such that $A_\ell, A_r, A$ are pairwise disjoint. Fix two bijections $\varphi_\ell \colon A \to A_\ell$ resp. $\varphi_r \colon A \to A_r$, and extend them to isomorphisms $\varphi_\ell$ resp. $\varphi_r$ of the corresponding free monoids $A^\ast$ and $A_\ell^\ast$ resp. $A_r^\ast$. If $\cR \subseteq A^\ast \times A^\ast$ is a rewriting system, then we let $\cR_\ell$ denote the system with rules $(\varphi_\ell(r), \varphi_\ell(s))$ whenever $(r, s) \in \cR$. We define $\cR_r$ analogously. It is readily seen that if $\cc$ is a class of languages closed under homomorphism, then $\cR$ is $\cc$-ancestry preserving if and only if $\cR_\ell$ (or indeed $\cR_r$) is $\cc$-ancestry preserving. 

We define a rational transduction. For ease of notation, let $A_\# = A \cup \{ \# \}$, and let $A_{\ell,r,\#} = A_\# \cup A_\ell \cup A_r$. Let $\mu_{\ell,r} \subseteq A_\#^\ast \times A_{\ell,r,\#}^\ast$ be defined by
\begin{equation}\label{Eq:Transduction_ell_to_r}
\mu_{\ell,r} = \bigg( \bigcup_{a \in A} (a, \varphi_\ell(a)) \bigg)^\ast (\#,\#) \bigg( \bigcup_{a \in A} (a, \varphi_r(a)) \bigg)^\ast.
\end{equation}
As $\mu_{\ell,r} = X^\ast x Y^\ast$ for finite subsets $X, Y \subseteq A_\#^\ast \times A_{\ell,r,\#}^\ast$ and a single element $x \in A_\#^\ast \times A_{\ell,r,\#}^\ast$, it follows that $\mu_{\ell,r}$ is a rational subset of $A_\#^\ast \times A_{\ell,r,\#}^\ast$, so it is a rational transduction. Furthermore, it is clear that if $L \subseteq A^\ast \# A^\ast$, then 
\begin{equation}\label{Eq:Mu_on_L}
\mu_{\ell,r}(L) = \{ \varphi_\ell(u) \# \varphi_r(v) \mid u \# v \in L\}.
\end{equation}
In particular, $\mu_{\ell,r}$ is bijective on subsets of $A^\ast \# A^\ast$. Analogously, $\mu_{\ell,r}^{-1}$ is bijective on subsets of $A_\ell^\ast \# A_r^\ast$ (here, by $\mu_{\ell,r}^{-1}$ we mean the usual inverse of a transduction). We remark that the notation $\mu_{\ell,r}$ is slightly abusive, as it suppresses any reference to the set $A$ or the symbol $\#$, or to the fixed bijections $\varphi_\ell, \varphi_r$. Context, however, will always make the former clear, and the choices of the latter make no difference, as we will not use any particular properties of our fixed bijections. 

\begin{proposition}\label{Prop:ancestry_main_prop}
Let $\cc$ be a class of languages closed under rational transductions. Let $L \subseteq A^\ast \# A^\ast$, and let $\cR_1, \cR_2 \subseteq A^\ast \times A^\ast$ be rewriting systems. If $\cR_1, \cR_2$ are $\cc$-ancestry preserving, then $L \in \cc \implies L^{\cR_1, \cR_2} \in \cc$.
\end{proposition}
\begin{proof}
Let $L_0 = \mu_{\ell,r}^{-1}\big( \left( \mu_{\ell,r}(L)\right)^{\cR_{1,\ell}, \cR_{2,r}} \big)$. We claim that $L^{\cR_1, \cR_2} = L_0$. This would imply the proposition by the following reasoning: first $\mu_{\ell,r}(L) \subseteq A_\ell^\ast \# A_r^\ast$ is in $\cc$ by closure of $\cc$ under rational transduction. Second, $\cR_{1,\ell}$ resp. $\cR_{2,r}$ is clearly $\cc$-ancestry preserving, as $\cR_1$ resp. $\cR_2$ is. As $A_\ell \cap A_r = \varnothing$, it follows from Lemma~\ref{Lem:Anc_partial_pres} that $\left( \mu_{\ell,r}(L)\right)^{\cR_{1,\ell}, \cR_{2,r}} \in \cc$. Finally, by closure under rational transduction (and as the inverse of a rational transduction is rational), applying $\mu_{\ell,r}^{-1}$ yields the result.

We prove the claim. First, $w \equiv w_1 \# w_2 \in L^{\cR_1, \cR_2}$ if and only if: (i) there exists $u_1 \# u_2 \in L$ with $w_i \xra{\cR_i} u_i$, with $u_i \in A^\ast$, for $i=1,2$. By \eqref{Eq:Mu_on_L}, it follows that (i) is equivalent to: (ii) there exists $v \in \mu_{\ell,r}(L)$ such that $v \equiv \varphi_\ell(u_1) \# \varphi_r(u_2)$ and $w_1 \xra{\cR_1} u_1$ and $w_2 \xra{\cR_2} u_2$. Note that for any $x, y \in A^\ast$, we have $x \xra{\cR_1} y$ if and only if $\varphi_\ell(x) \xra{\cR_{1,\ell}} \varphi_\ell(y)$, and analogously for $\cR_2, \varphi_r$ and $\cR_{2,r}$. Hence using $u_1 \equiv \varphi_\ell^{-1}(\varphi_\ell(u_1))$ and $u_2 \equiv \varphi_r^{-1}(\varphi_r(u_2))$, it follows that (ii) is equivalent to: (iii) there exists $v \in \mu_{\ell,r}(L)$ such that $v \equiv \varphi_\ell(u_1) \# \varphi_r(u_2)$ and $\varphi_\ell(w_1) \xra{\cR_{1,\ell}} \varphi_\ell(u_1)$ and $\varphi_r(w_2) \xra{\cR_{2,r}} \varphi_r(u_2)$. But now, by \eqref{Eq:Bipartisan_ancestor_def}, (iii) is simply equivalent to: (iv) $\varphi_\ell(w_1) \# \varphi_r(w_2) \in \left( \mu_{\ell,r}(L) \right)^{\cR_{1,\ell}, \cR_{2,r}}$. As $\varphi_{\ell}(w_1) \# \varphi_r(w_2) \equiv \mu_{\ell,r}(w)$, it follows by the bijectivity of $\mu_{\ell,r}^{-1}$ on subsets of $A_\ell^\ast \# A_r^\ast$ that (iv) is finally equivalent to: (v) $w \in L_0$. This is what was to be shown. 
\end{proof}

We have the following corollary, being easier to state than Proposition~\ref{Prop:ancestry_main_prop}.

\begin{proposition}\label{Prop:LinC=>L(R1,R2)inC,monadic_prop}
Let $\cc$ be a super-$\AFL$. Let $L \subseteq A^\ast \# A^\ast$, and let $\cR_1, \cR_2$ be monadic $\cc$-rewriting systems. Then $L \in \cc \implies L^{\cR_1, \cR_2} \in \cc$.
\end{proposition}

We are now ready to apply the theory of bilateral ancestors to free products; we shall see that (the word problem for) the monoid free product is obtained as a bilateral ancestor of the semigroup free product, which in turn is obtained as an alternating free product of the word problems of the free factors.

\section{Free products and word problems}\label{Sec:FPandWP}

\noindent In this section, we will describe the language-theoretic  properties of free products of semigroups and monoids, using the tools developed in \S\ref{Sec:Toolbox}. The overall, informal idea is the following: for both semigroup and monoid free products, the alternating product of the word problems of the free factors forms the first layer. For semigroup free products, this is in fact a complete description. For monoid free products, on the other hand, we also need to use bipartisan ancestors to account for the fact that not all words are reduced. 

\subsection{Semigroup free products}\label{Subsec:WP_Semigroup}

All the legwork for describing semigroup free products is performed by alternating products. 

\begin{lemma}\label{Lem:WP_sfp_is_AP}
Let $S_1, S_2$ be two semigroups, generated by disjoint sets $A_1, A_2$, respectively. Let $S_1 \ast S_2$ denote their semigroup free product. Then 
\[
\WP_{A_1 \cup A_2}^{S_1 \ast S_2} = \WP_{A_1}^{S_1} \star \WP_{A_2}^{S_2}.
\]
\end{lemma}
\begin{proof}
First, as $\WP_{A_1}^{S_1}$ and $\WP_{A_2}^{S_2}$ are concatenation-closed languages, their alternating product $\WP_{A_1}^{S_1} \star \WP_{A_2}^{S_2}$ is well-defined. An arbitrary word $w$ is in $\WP_{A_1 \cup A_2}^{S_1 \ast S_2}$ if and only if $w \equiv u \# v^\trev$ with $u, v  \in (A_1 \cup A_2)^+$ and $u =_S v$. By Lemma~\ref{Lem:Semigroup_free_product_normalforms}, there exist unique factorisations 
\begin{equation*}
u \equiv u_0 u_1 \cdots u_n \quad \textnormal{and} \quad v \equiv v_0 v_1 \cdots v_n,
\end{equation*}
such that $u_i, v_i \in A_{X(i)}^+$ and $u_i =_{S_{X(i)}} v_i$ for all $0 \leq i \leq n$, where $X$ is a standard parametrisation. But this is equivalent to 
\begin{equation}\label{Eq:SGP_FP_AP}
w \equiv u_1 u_2 \cdots u_n \# (v_1 v_2 \cdots v_n)^\trev \equiv u_1 u_2 \cdots u_n \# v_n^\trev \cdots v_2^\trev v_1^\trev,
\end{equation}
such that $u_i \# v_i^\trev \in \WP_{A_{X(i)}}^{S_{X(i)}}$ for $0 \leq i \leq n$. By the definition \eqref{Eq:Def_alternating} of the alternating product, \eqref{Eq:SGP_FP_AP} is equivalent to $w \in  \WP_{A_1}^{S_1} \star \WP_{A_2}^{S_2}$, as desired. 
\end{proof}

We conclude the following by combining Proposition~\ref{Prop:L1,L2=>L1starL2_in_C} and Lemma~\ref{Lem:WP_sfp_is_AP}:

\begin{theorem}\label{Theo:SGP_closed_under_FP}
Let $\cc$ be a super-$\AFL$. Then the class of semigroups with word problem in $\cc$ is closed under taking semigroup free products. 
\end{theorem}

As a very particular case, taking $\cc = \CF$, the class of context-free languages, we recover the following recent result.

\begin{corollary}[Brough, Cain \& Pfeiffer, 2019]\label{Cor:SGP_CF_closed_under_FP}
The class of semigroups with context-free word problem is closed under taking semigroup free products.
\end{corollary}

We note that the method (pushdown automata) used by the authors of \cite{Brough2019} to prove Corollary~\ref{Cor:SGP_CF_closed_under_FP} is very different from the method used in this paper.  Given that the class $\IND$ of indexed languages is a super-$\AFL$, we find:

\begin{corollary}\label{Cor:SGP_IND_closed_udner_FP}
The class of semigroups with indexed word problem is closed under taking semigroup free products.
\end{corollary}

We now turn to prove the analogous statements for monoid free products.

\subsection{Monoid free products}

The identification of the identity elements of the factors in a monoid free product means that the word problem of the product will not, in general, be the alternating product of the factors' word problems. 

\begin{example}
Let $M_1 = \pres{Mon}{a_1}{a_1^2=1}$ and $M_2 = \pres{Mon}{a_2}{a_2^2=1}$, and let $A_1 = \{ a_1 \}$, $A_2 = \{ a_2 \}$, $A = A_1 \cup A_2$. Let $M = M_1 \ast M_2$ denote their monoid free product, isomorphic to the infinite dihedral group $D_\infty$. Then $a_1 a_2^2 a_1 =_{M} 1$, but the element $a_1 a_2^2 a_1 \# \varepsilon$ is not an element of the alternating product $\WP_{A_1}^{M_1} \star \WP_{A_2}^{M_2}$, even though $a_1 a_2^2 a_1 \#\varepsilon \in \WP_A^M$.
\end{example}

Recall the definition \eqref{Eq:IP_Defn} of $\IP_A^M$. The following is more or less obvious.

\begin{lemma}\label{Lem:WP=>IP}
Let $\cc$ be a class of languages closed under quotient with regular languages. Let $M$ be a monoid, finitely generated by $A$. If $\WP_A^M \in \cc$, then $\IP_A^M \in \cc$.
\end{lemma}
\begin{proof}
The result immediately follows from 
\[
\IP_A^M  = \{ w \mid w \# \varepsilon \in \WP_A^M \} = \WP_A^M / \{ \# \varepsilon \}.
\]
\end{proof}

The converse of Lemma~\ref{Lem:WP=>IP} only rarely holds; however, when all defining relations of $M$ are of the form $w = 1$ (i.e. when $M$ is \textit{special}), then it does hold \cite{NybergBrodda2020b}. The following is a key lemma.

\begin{lemma}\label{Lem:WP_Monoids_is_bipartisan}
Let $M_1, M_2$ be two monoids generated by finite sets $A_1, A_2$ respectively. Let $M = M_1 \ast M_2$ be their monoid free product, and $A = A_1 \cup A_2$. Let 
\begin{align*}
\cR_1 &= \left\{ (w, 1) \mid w \in \IP_{A_1}^{M_1} \cup \IP_{A_2}^{M_2}\right\} \quad \textnormal{and} \quad  \cR_2 = \left\{ (w^\trev, 1) \mid w \in \IP_{A_1}^{M_1} \cup \IP_{A_2}^{M_2}\right\}.
\end{align*}
Then we have 
\begin{equation}\label{Eq:WP_monoid_equality}
\WP_{A}^{M} = \big( \WP_{A_1}^{M_1} \star \WP_{A_2}^{M_2} \big)^{\cR_1, \cR_2}.
\end{equation}
\end{lemma}
\begin{proof}
For ease of notation, let $\cW_i = \WP_{A_i}^{M_i}$ for $i=1,2$, and let $\cW = \WP_{A}^M$. As both $\cW_1$ and $\cW_2$ are concatenation-closed, their alternating product $\cW_1 \star \cW_2$ is well-defined. Let $L_0 = (\cW_1 \star \cW_2)^{\cR_1, \cR_2}$ denote the right-hand side of \eqref{Eq:WP_monoid_equality}. We prove the claimed equality \eqref{Eq:WP_monoid_equality}, i.e. $\cW = L_0$.

$(\supseteq)$ Suppose $w \in L_0$. Then by \eqref{Eq:Bipartisan_ancestor_def}, we have $w \equiv w_1 \# w_2$ with $w_1, w_2 \in A^\ast$, such that there exist $u_1, u_2 \in \cW_1 \star \cW_2$ such that $w_i \xra{\cR_1} u_i$ for $i=1,2$. As $\cR_1$ is $M_1$-equivariant and $M_2$-equivariant, it is also $M$-equivariant. In particular $w_1 =_M u_1$. As $\cR_2$ is $M_1^\trev$-equivariant and $M_2^\trev$-equivariant, it is $M^\trev$-equivariant. In particular, $w_2^\trev =_{M^\trev} u_2^\trev$. As $u_1 \# u_2 \in \cW_1 \star \cW_2$, it follows easily, just as in the proof of Lemma~\ref{Lem:WP_sfp_is_AP}, that $u_1 \# u_2 \in \cW$. Thus $u_1 =_M u_2^\trev$, so 
\[
w_1 =_M u_1 =_M u_2^\trev =_M w_2^\trev,
\]
and hence $w_1 \# w_2^\trev \in \cW$, as desired.

$(\subseteq)$ Suppose $w \in \cW$. Then $w \equiv u \# v^\trev$, with $u,v \in A^\ast$ and $u =_M v$. Let $u'$ be $\cR_1$-irreducible and such that $u \xra{\cR_1} u'$. Let $v'$ be such that $v'$ is $\cR_1$-irreducible, and such that $v \xra{\cR_1} v'$. Then $v^\trev \xra{\cR_2} (v')^\trev$. Now $u'$ resp. $v'$ is a reduced form of $u$ resp. $v$. Let the corresponding normal forms be
\[
u' \equiv u_0 u_1 \cdots u_m \textnormal{ resp. } v' \equiv v_0 v_1 \cdots v_n.
\]
By Lemma~\ref{Lem:Monoid_free_product_normal_form}, as $u =_M v$ we have $n=m$, $u_i, v_i \in A_{X(i)}$ and $u_i =_{M_{X(i)}} v_i$ for every $0 \leq i \leq n$, where $X$ is a standard parametrisation. Thus, for every $0 \leq i \leq n$, we have $u_i \# v_i^\trev \in \cW_{X(i)}$. Furthermore, as 
\[
u' \# (v')^\trev \equiv u_0 u_1 \cdots u_m \# (v_0 v_1 \cdots v_n)^\trev \equiv u_0 u_1 \cdots u_m \# v_n^\trev \cdots v_2^\trev v_1^\trev,
\]
it follows by the definition \eqref{Eq:Def_alternating} of alternating products that 
\[
u' \# (v')^\trev \in \cW_1 \star \cW_2.
\]
As $u \xra{\cR_1} u'$ and $v^\trev \xra{\cR_2} (v')^\trev$, thus $u \# v^\trev \in (\cW_1 \star \cW_2)^{\cR_1, \cR_2}$. Or, in other words, $w \in L_0$.
\end{proof}

Thus, using the results from \S\ref{Sec:Toolbox}, we now easily deduce:

\begin{theorem}\label{Thm:Monoid_FP_WP}
Let $\cc$ be a super-$\AFL$ closed under reversal. Then the class of monoids with word problem in $\cc$ is closed under taking monoid free products.
\end{theorem}
\begin{proof}
Let $M_1, M_2$ be two monoids generated by finite disjoint sets $A_1, A_2$. Let $M = M_1 \ast M_2$ denote their free product generated by $A = A_1 \cup A_2$. Let $\cW_i = \WP_{A_i}^{M_i}$ for $i=1,2$, and let $\cW = \WP_{A}^{M}$. Let $\cI_i = \IP_{A_i}^{M_i}$ for $i=1,2$. Assume $M_1$ and $M_2$ have word problem in $\cc$. Then $\cW_1, \cW_2 \in \cc$. By Lemma~\ref{Lem:WP=>IP}, also $\cI_1, \cI_2 \in \cc$. Then $\cR_1, \cR_2$, as in the statement of  Lemma~\ref{Lem:WP_Monoids_is_bipartisan}, are both monadic $\cc$-rewriting systems, as $\cc$ is closed under unions and reversal. In particular, $\cR_1, \cR_2$ are $\cc$-ancestry preserving. By Lemma~\ref{Lem:WP_Monoids_is_bipartisan}, $\cW = (\cW_1 \star \cW_2)^{\cR_1, \cR_2}$. Now, by Proposition~\ref{Prop:L1,L2=>L1starL2_in_C}, $\cW_1 \star \cW_2 \in \cc$. As $\cR_1, \cR_2$ are $\cc$-ancestry preserving, we thus conclude $\cW \in \cc$, as was to be shown. 
\end{proof}

As in the case of semigroup free products, we recover the following, due to Brough, Cain \& Pfeiffer \cite{Brough2019}; however, the authors only provide a sketch proof. 

\begin{corollary}[Brough, Cain \& Pfeiffer, 2019]
The class of monoids with context-free word problem is closed under taking monoid free products.
\end{corollary}

As the monoid free product of groups coincides with the (group) free product of groups, we find the following corollary, of independent interest for group theorists.

\begin{corollary}\label{Cor:WP_groups_closed_FP}
Let $\cc$ be a super-$\AFL$ closed under reversal. Then the class of groups with word problem in $\cc$ is closed under taking free products.
\end{corollary}

Of course, for $\cc = \CF$, the class of context-free languages, Corollary~\ref{Cor:WP_groups_closed_FP} is well-known, and was proved already by An\={\i}s\={\i}mov \cite{Anisimov1971}. Indeed,  by the Muller-Schupp theorem, it then says little more than the fact that the free product of two virtually free groups is again virtually free, which is easy to prove directly. On the other hand, Corollary~\ref{Cor:WP_groups_closed_FP} specialises to the following result: 

\begin{corollary}\label{Cor:Indexed_groups_WP}
The class of groups with indexed word problem is closed under taking free products.
\end{corollary}

In general, the problem of determining what the class of groups with indexed word problem is a wide open problem; it is not known to be strictly larger than the class of context-free groups. Indeed, it is not even known whether $\mathbb{Z}^2$ has indexed word problem \cite{Gilman2018}. The above Corollary~\ref{Cor:Indexed_groups_WP} does not appear to have been demonstrated anywhere in the literature on indexed groups. Note that the ``word problem'' in the statement of Corollary~\ref{Cor:Indexed_groups_WP} can either be taken to be $\WP$ or $\IP$, i.e. the set of words equal to $1$, the latter being more common in group theory. 

It seems a (Sisyphean) task to determine precisely the classes of languages $\cc$ such that the class of semigroups (or monoids) with word problem in $\cc$ is closed under free products, outside the pleasant setting of super-$\AFL$s. For example, the class $\REG$ of regular languages certainly does not have either of these two properties, as the free product of finite semigroups/monoids is not generally finite. On the other hand, it is an open problem (see \cite{Brough2019}) whether the class of semigroups (or monoids) with word problem in $\DCF$, the deterministic context-free languages, is closed under free products. We contribute some questions to this general line of inquiry. 

Let $\cc$ be a class of languages closed under inverse homomorphism. We say that $\cc$ is $\operatorname{SFP}$ if the class of semigroups with word problem in $\cc$ is closed under semigroup free products. Analogously, we say that $\cc$ is $\operatorname{MFP}$ if the class of monoids with word problem in $\cc$ is closed under monoid free products. The property $\operatorname{GFP}$, for groups, is defined analogously. Of course, as observed earlier, $\operatorname{MFP} \implies \operatorname{GFP}$. It is natural to ask whether the converse is true, which seems a difficult problem. The relation between $\operatorname{MFP}$ and $\operatorname{SFP}$ also seems rather reticent.

\begin{question}\label{Quest:MFP_versus_SFP}
Is there a class $\cc$ of languages such that $\cc$ has $\operatorname{SFP}$, but not $\operatorname{MFP}$? Or, indeed, such that it has $\operatorname{MFP}$, but not $\operatorname{SFP}$? Such that it has $\operatorname{GFP}$, but not $\operatorname{SFP}$?
\end{question}

We remark that the semigroup free product of two trivial semigroups is infinite, and therefore not regular (with respect to its word problem); whereas the monoid free product of two trivial monoids is trivial, and therefore regular. Thus, there is already a small distinction between the language-theoretic behaviours of the two classes. More non-trivially, Question~\ref{Quest:MFP_versus_SFP} encodes two counteracting tensions: on the one hand, the bipartisan ancestors used for monoid free products seems able to create rather complex behaviour in the word problem for the monoid free product; but on the other hand, the particular form of word problem languages means that any such complex behaviour is rather restricted. To the author, it certainly seems conceivable that there is a positive answer to (some part of) Question~\ref{Quest:MFP_versus_SFP}, by way of some cleverly chosen examples.

\section*{Acknowledgements}

\noindent Most of the research in this article was carried out while the author was a Ph.D. student at the University of East Anglia, and the main results appear in the author's Ph.D. thesis \cite{Thesis}. The author wishes to thank his supervisor Robert D. Gray for much helpful feedback. Finally, the author wishes to thank the Dame Kathleen Ollerenshaw Trust for funding his current research at the University of Manchester.

\bibliography{wp_free_products} 
\bibliographystyle{amsalpha}

\end{document}